\documentclass[10pt]{amsart} 
\usepackage{amsmath,amssymb}
\usepackage[dvipdfmx]{graphicx}
\usepackage{float}
\usepackage{multicol}
\usepackage{ascmac}
\usepackage{mathtools}
\usepackage{bm}
\usepackage{amscd}
\usepackage{comment}
\usepackage{mathrsfs}
\usepackage[all]{xy}
\usepackage{color}

\usepackage{cases}
\usepackage{paralist}
\title[Congruences of the $q$-Fibonacci sequence]{Congruences of the $q$-Fibonacci sequence related with its transcendence}
\author{Takumi Anzawa\and
Hidetaka Funakura}
\email{m20001s@math.nagoya-u.ac.jp , hidetaka.funakura@gmail.com}
\address{graduate school of mathematics, nagoya university, furo-cho, chikusa-ku, nagoya, 464-8602, japan}
\date{\today} 
\usepackage[margin=35mm]{geometry}
{\theoremstyle{definition} \newtheorem{df}{Definition}[section]}%[subsection]
{\theoremstyle{definition}\newtheorem{ex}[df]{Example}}
{\theoremstyle{definition}\newtheorem{prop}[df]{Proposition}}
{\theoremstyle{definition}\newtheorem{lem}[df]{Lemma}}
{\theoremstyle{definition}}
{\theoremstyle{definition} \newtheorem{rmk}[df]{Remark}}
\newtheorem{thm}[df]{Theorem}
{\theoremstyle{definition}\newtheorem{conj}[df]{Conjecture}}

\DeclareMathOperator{\id}{id}

\DeclareMathOperator{\Li}{Li}
\DeclareMathOperator{\Gal}{Gal}

\DeclareMathOperator{\ord}{ord}

\newcommand{\ca}[1]{\mathscr{#1}}

\newcommand{\qbinom}[2]{\left[ #1 \atop #2 \right]}

\newcommand{\A}{\mathscr{A}}
\newcommand{\fz}{\zeta_{\A}}

\newcommand{\engMonth}{
	\ifcase\month
		\or January 
		\or February 
		\or March 
		\or April
		\or May 
		\or June
		\or July 
		\or August 
		\or September
		\or October 
		\or November 
		\or December\fi
} 
\newcommand{\fullday}{
	\engMonth \, \the\day, \the\year	
} 

\newcommand{\Z}{\mathbb{Z}}
\newcommand{\Q}{\mathbb{Q}}
\newcommand{\Qi}{\mathcal{C}_{\A}^0}

\newcommand{\resi}[2]{\left(\frac{#1}{#2}\right)}
\newcommand{\Prime}{P}

\newcommand{\floor}[1]{\lfloor #1 \rfloor}
\newcommand{\lsym}[1]{\left(\frac{#1}{5}\right)}

\numberwithin{equation}{section}

\allowdisplaybreaks[1]
\begin{document}
\maketitle
\begin{abstract}
By using Andrews' explicit formulae of the $q$-Fibonacci sequence introduced by Schur, we prove certain congruences of the $q$-Fibonacci sequence which relate the sequence with the original Fibonacci sequence.
As a corollary, we show that it yields a transcendental element in the $\mathbb{Q}$-algebra $\A$ of integers modulo infinitely large primes under the generalized Riemann hypothesis.
\end{abstract}
\tableofcontents

%%%%%%%%%%%%%%%%%%%%%%%%%%%%%%%%Introduction%%%%%%%%%%%%%%%%%%%%%%%%%%%%%%%%%%%%

\section{Introduction}
In 1917, Schur (\cite{Schur}) introduced the so-called the $q$-\textit{Fibonacci sequence} $\{F_n(q)\}$ which is the sequence of $\mathbb{Q}[q]$ defined by the initial value $(F_0(q),F_1(q))=(0,1)$ and the recurrence relation
\[
F_{n+2}(q)-F_{n+1}(q)-q^nF_n(q)=0
\]
for every $n\in\mathbb{N}$. It recovers the ordinary Fibonacci sequence $\{F_n\}$ when $q=1$.
Andrews (\cite{andrews}) gave an explicit formula (cf. Theorem \ref{q-Fibonacci}) of the $q$-Fibonacci sequence to prove some kind of the Rogers-Ramanujan identities.

Let $P$ be the set of prime numbers and
let $v_p(\alpha)$ denote the $p$-adic valuation of $\alpha$ for $\alpha\in\mathbb{Q}^\times$ and $p\in P$ and set $v_p(0)=0$.
For a pair $(\alpha,p)\in\mathbb{Q}^\times\times P$ with $v_p(\alpha)=0$,
$\ord_{p}(\alpha)$ denotes the order of $\alpha$ in the multiplicative group $(\mathbb{Z}/p\mathbb{Z})^\times$ and $I_p(\alpha):=(p-1)/\ord_{p}(\alpha)$, i.e. $I_p(\alpha)$ is the index of the subgroup of $(\mathbb{Z}/p\mathbb{Z})^\times$ generated by $\alpha$.
So when $\alpha$ is a primitive root, we have $\ord_{p}(\alpha)=p-1$ and $I_p(\alpha)=1$.
The values $\ord_p(\alpha)$ and $I_p(\alpha)$ are called the \textit{residual order} of $\alpha$ and the \textit{residual index} of $\alpha$ respectively.
Our main theorem is on congruence which relates the $q$-Fibonacci sequence with the ordinary one:

%%%%%%%%%%%%%%%%% First main theorem
\begin{thm}\label{thm-q-fib}
For $\alpha\in \mathbb{Q}^\times$ and $p\in P$ satisfying $v_p(\alpha)=v_p(\alpha-1)=0$ and $\ord_{p}(\alpha)\not\equiv 0 \mod 5$,
\[
F_{p}(\alpha)\equiv F_{I_p(\alpha)+\lsym{\ord_p(\alpha)}} \mod p
\]
holds, where $\lsym{m}$ is the Legendre symbol for any $m\in\mathbb{Z}$.
\end{thm}
The following quotient $\mathbb{Q}$-algebra
\[
\mathscr{A}:=\left(\prod_{p\in P}\mathbb{Z}/p\mathbb{Z}\right)\Bigm/\left(\bigoplus_{p\in P} \mathbb{Z}/p\mathbb{Z}\right)
\]
appeared in \cite{Kon}, has been studied in several kinds of literature (\cite{Rosen1}, \cite{rtty},   \cite{Seki}, etc) in relation with the study of finite multiple zeta values (FMZVs, in short) introduced by Kaneko and Zagier (\cite{Kan1}). 
Rosen (\cite{Rosen1}) introduced the notion of finite algebraic numbers in $\A$ by using recurrent sequences.
It should be noted that solutions of $\mathbb{Q}$-polynomials in $\A$ are not always finite algebraic numbers in $\A$ (in Rosen's sense).
In this paper, we consider \textit{$\A$-transcendental numbers}, elements in $\A$ which are not roots of non-zero $\mathbb{Q}$-polynomials in $\A$ (Definition \ref{df-ftn}).
It is expected that non-zero FMZVs are $\A$-transcendental,
however, so far no single example has been obtained as far as authors know.
Our interests in this paper are to construct examples of $\A$-transcendental numbers in the algebra $\A$

We show that the $(F_p(\alpha))_p\in \A$ is an $\A$-transcendental number by combining Theorem \ref{thm-q-fib} and Moree's result (\cite{MoreeII}) on the density of certain primes related to the residual index.

\begin{thm}\label{thm-trans}
Under the generalized Riemann hypothesis (GRH, in short), $(F_{p}(g))_{p}\in \A$ is an $\A$-transcendental number when $g\in\mathbb{Z}_{>1}$ is square-free.
\end{thm}
%%%%%%%%%%%%%%%%%

%\subsection{A root of a rational irreducible polynomial in $\A$}
%\subsection{The structure of $\Qi$}

%%%%%%%%%%%%%%%%%%%%%%%%%%%%%%%%Section2%%%%%%%%%%%%%%%%%%%%%%%%%%%%%%%%%%%%%%

\section{Congruences of $q$-Fibonacci sequence}
In this section, we prove our main theorem (Theorem \ref{thm-q-fib}).

%%%%%%%%%%%%%%%%%%%%%%%%%%%%%%%%Subsection2.1%%%%%%%%%%%%%%%%%%%%%%%%%%%%%%%%%%

\subsection{Review on $q$-analogues}
\verb| |

We recall the following standard notation:
\begin{itemize}
\item For $n\in\mathbb{Z}_{>0}$, the $q$-integer $[n]_q$ is defined by 
\[
\displaystyle [n]_q:=\frac{1-q^n}{1-q}.
\]
\item For $n\in\mathbb{Z}_{>0}$, the $q$-factorial $[n]_q!$ is defined by 
\[
\displaystyle [n]_q!:=[n]_q[n-1]_q\cdots [1]_q.
\]
\item For a pair of integers $n$, $m$, the $q$-binomial coefficient $\displaystyle \qbinom{n}{m}_q$ is defined as follows: 
\[
\displaystyle \qbinom{n}{m}_q:=
\begin{cases}
\frac{[n]_q\cdots [n-m+1]_q}{[m]_q\cdots [1]_q} &\text{if }0\le m\le n\\
0&\text{else}.
\end{cases}
\]
\end{itemize}
Andrews gave a general explicit formula of the $q$-Fibonacci sequence.
\begin{thm}[{\cite[Theorem]{andrews}}]\label{q-Fibonacci}
For any non-negative integer $n$,
\begin{align}
F_{n+1}(q)
&=\sum_{j=-\infty}^\infty (-1)^j q^{j(5j+1)/2} \qbinom{n}{\lfloor (n-5j)/2 \rfloor}_q
\end{align}
holds, 
where  $\lfloor x\rfloor$ denotes the greatest integer not exceeding $x$.
\end{thm}

%%%%%%%%%%%%%%%%%%%%%%%%%%%%%%%%Subsection2.2%%%%%%%%%%%%%%%%%%%%%%%%%%%%%%%%%%

\subsection{Preliminaries}\verb| |

This subsection presents some lemmas of $q$-integers and $q$-binomial coefficients.
We show the prime congruence of the $q$-Fibonacci sequence by the lemmas.
Let $\mathbb{Z}_{(p)}$ be the localization of $\mathbb{Z}$ with respect to the prime ideal $(p)$ generated by a prime number $p$. 
Note that the isomorphism of field $\mathbb{Z}/p\mathbb{Z}\cong \mathbb{Z}_{(p)}/p \mathbb{Z}_{(p)}$ holds.
%\begin{shadebox}
\begin{lem}
\label{lem-qint}
Let $\alpha \in \Q\setminus\{0,1\}$
and $p \in \Prime$ with $v_p(\alpha)=v_p(\alpha-1)=0$.
Let $l$ and $k\in \Z$ be integers satisfying $1\le l\le p-1$ and $k \equiv l \mod \ord_{p}(\alpha)$. Then $\displaystyle \frac{[k]_{\alpha}}{[l]_{\alpha}} \in\mathbb{Z}_{(p)}\setminus\{0\}$ and 
\begin{equation}
	\label{eq-qint} 
	\frac{[k]_{\alpha}}{[l]_{\alpha}} 
	\equiv
\left\{
	\begin{array}{lll}
		\frac{k}{l} &\mod p& \text{if } k \equiv 0 \mod \ord_{p}(\alpha) \\
		1 &\mod p& \text{if }k \not\equiv 0\mod \ord_{p}(\alpha) \\
	\end{array}
\right.
\end{equation}
holds.
\end{lem}
\begin{rmk}
Note that for $n\in\mathbb{Z}$, $\alpha \in \Q\setminus\{0,1\}$ and $p$ be a prime number satisfying $v_p(\alpha)=v_p(\alpha-1)=0$, $[n]_{\alpha}\equiv 0 \mod p$ holds if and only if $n\equiv 0 \mod \ord_p(\alpha)$. We omit the proof.
% or $\alpha-1 \in p\mathbb{Z}_{(p)}$ and $n\equiv 0 \mod p$ hold.

\end{rmk}

%\end{shadebox}
\begin{proof}[Proof of Lemma \ref{lem-qint}]

Let $i$ be an integer satisfying $0\le i\le \ord_p(\alpha)-1$ and $k\equiv i \mod \ord_{p}(\alpha)$.
If $i\neq 0$, then $[l]_{\alpha}\not\equiv 0 \mod p$. We have
\begin{equation*}
	\frac{[k]_{\alpha}}{[l]_{\alpha}} 
	= 
	\frac{\alpha^{k - i}\alpha^i - 1}{\alpha^{l - i}\alpha^i - 1}
	\equiv 
	\frac{\alpha^i - 1}{\alpha^i - 1}
	=
	1 \mod p.
\end{equation*}
If $i=0$, we have
\begin{align*}
\left.\frac{[k]_{q}}{[l]_{q}}\right|_{q=\alpha}
&=\left.\frac{q^{k} - 1}{q^{l} - 1}\right|_{q=\alpha}=
\left.\frac{(1-q^{\ord_p(\alpha)})\sum_{s=0}^{k/\ord_p(\alpha)-1} q^{s\ord_p(\alpha)}}{(1-q^{\ord_p(\alpha)})\sum_{t=0}^{l/\ord_p(\alpha)-1} q^{t\ord_p(\alpha)}}\right|_{q=\alpha}
=\frac{\sum_{s=0}^{k/\ord_p(\alpha)-1} \alpha^{s\ord_p(\alpha)}}{\sum_{t=0}^{l/\ord_p(\alpha)-1} \alpha^{t\ord_p(\alpha)}}\\
&\equiv
\frac{k/\ord_p(\alpha)}{l/\ord_p(\alpha)}\mod p\\
&=\frac{k}{l} .
\end{align*}
\end{proof}
\begin{lem}\label{lem-p-1}
Let $\alpha\in\mathbb{Q}\setminus\{0,1\}$ and $p\in P$ satisfying $v_p(\alpha)=v_p(\alpha-1)=0$. Let $k$ be a positive integer with $0\le k\le p-1-\ord_p(\alpha)$. Put $\displaystyle C_k:=\frac{[p-k-1]_{\alpha}\cdots [p-k-\ord_p(\alpha)]_{\alpha}}{[k+\ord_p(\alpha)]_{\alpha}\cdots [k+1]_{\alpha}}$.
Then $C_k\in \mathbb{Z}_{(p)}\setminus\{0\}$ and 
\[
\qbinom{p-1}{k+\ord_{p}(\alpha)}_{\alpha}\equiv C_k\qbinom{p-1}{k}_{\alpha} \mod p
\]
holds.
Especially if there exists $l\in\mathbb{Z}$ such that $k=l\ord_p(\alpha)$, then we have
$\displaystyle C_{l\ord_p(\alpha)}\equiv\frac{I_p(\alpha)-l}{l+1} \mod p$.
\end{lem}

\begin{proof}
By the definition, we have
\begin{align*}
\qbinom{p-1}{k+\ord_p(\alpha)}_{\alpha}&=
\frac{[p-1]_{\alpha}\cdots [p-k-\ord_{p}(\alpha)]_{\alpha}}{[k+\ord_p(\alpha)]_{\alpha}\cdots [1]_{\alpha}}\\
&= \frac{[p-k-1]_{\alpha}\cdots [p-k-\ord_p(\alpha)]_{\alpha}}{[k+\ord_p(\alpha)]_{\alpha}\cdots [k+1]_{\alpha}}\times 
\frac{[p-1]_{\alpha}\cdots [p-k]_{\alpha}}{[k]_{\alpha}\cdots [1]_{\alpha}}\\
&=C_k \qbinom{p-1}{k}_{\alpha}.
\end{align*}
Note that for any $1\le i \le \ord_p(\alpha)$, there exists unique $j_i\in\{1,\ldots,\ord_p(\alpha)\}$ such that $k+i\equiv p-k-j_i\mod \ord_p(\alpha)$. This correspondence is one-to-one.
We have
\begin{align*}
C_k&= \frac{[p-k-1]_{\alpha}\cdots [p-k-\ord_p(\alpha)]_{\alpha}}{[k+\ord_p(\alpha)]_{\alpha}\cdots [k+1]_{\alpha}}\\
&=\frac{[p-k-j_1]_{\alpha}}{[k+1]_{\alpha}}\cdot \frac{[p-k-j_2]_{\alpha}}{[k+2]_{\alpha}}\cdots \frac{[p-k-j_{\ord_p(\alpha)}]_{\alpha}}{[k+\ord_p(\alpha)]_{\alpha}}.
\end{align*}
By $1\le k+1,\ldots,k+\ord_p(\alpha)\le p-1$,
Lemma \ref{lem-qint} implies $\displaystyle \frac{[p-k-j_i]_{\alpha}}{[k+i]_{\alpha}}\in\mathbb{Z}_{(p)}\setminus\{0\}$ for every $1\le i\le \ord_p(\alpha)$ and $C_k\in \mathbb{Z}_{(p)}\setminus\{0\}$ holds.

In tha case where $k$ is given by $k=l\ord_p(\alpha)$,
Lemma \ref{lem-qint}  implies
\begin{align*}
C_{l\ord_p(\alpha)}\equiv& \frac{[p-k-1]_{\alpha}}{[k+\ord_p(\alpha)]_{\alpha}}\times \frac{[p-k-2]_{\alpha}}{[k+\ord_{p}(\alpha)-1]_{\alpha}}\times\cdots\times \frac{[p-k-\ord_p(\alpha)]_{\alpha}}{[k+1]_{\alpha}}\\
\equiv& \frac{p-k-1}{k+\ord_{p}(\alpha)}
\equiv \frac{I_p(\alpha)-l}{l+1} \mod p.
\end{align*}
\end{proof}
\begin{lem}\label{lem-qbinom2}
Let $\alpha\in\mathbb{Q}\setminus\{0,1\}$ and $p\in P$ satisfying $v_p(\alpha)=v_p(\alpha-1)=0$. Let $k$ be a positive integer satisfying $0\le k< \ord_p(\alpha)$. Then
\[
\qbinom{p-1}{k}_{\alpha}\equiv
\begin{cases}
1\mod p& \text{if } k=0\\
0\mod p& \text{if } k\neq 0
\end{cases}
\]
holds.
\end{lem}
\begin{proof}
We have a conclusion immediately if $k=0$.
Assume $1\le k< \ord_p(\alpha)$. Note that $[k]_{\alpha}!\not\equiv 0 \mod p$ holds since  $k < \ord_p(\alpha) $ holds. 
By $[p-1]_{\alpha}\equiv 0 \mod p$,  we have $\displaystyle\qbinom{p-1}{k}_{\alpha}\equiv 0 \mod p$.
\end{proof}

%\begin{shadebox}
\begin{lem}\label{lem-qbinom}
For $\alpha \in \Q\setminus\{0,1\}$
and $p \in \Prime$ with $v_p(\alpha)=v_p(\alpha-1)=0$,
\begin{align*}
	\qbinom{p - 1}{k}_{\alpha}
	&\equiv 
\left\{
	\begin{array}{ccl}
		\binom{I_{p}(\alpha)}{k/\ord_{p}(\alpha)} &\mod p & \text{if }k \in (\ord_{p}(\alpha)) \\
		0 &\mod p & \mathrm{else} \\
	\end{array}
\right.
\end{align*}
holds.
\end{lem}
%\end{shadebox}

\begin{proof}
The case of $k \not\in (\ord_{p}(\alpha))$ is clear by Lemma \ref{lem-p-1} and Lemma \ref{lem-qbinom2}. 
If we suppose $k=l\ord_{p}(\alpha)$, then we have
\begin{align*}
	\qbinom{p - 1}{k}_{\alpha}
	&\equiv
		\frac{
			(I_p(\alpha) - (l - 1))(I_p(\alpha) - (l - 2)) \cdots (I_p(\alpha) - 0)
		}{
			(l -1 + 1)(l - 2 + 1)\cdots(0 + 1)
		}
		\cdot 1 \\	
	&\equiv 
	\binom{I_{p}(\alpha)}{l} \mod p.
\end{align*}
\end{proof}

The following proposition is the prime congruence of the $q$-Fibonacci sequence.
%\begin{shadebox}
\begin{prop}
\label{prop-old-qfib-p}
For $\alpha \in \Q\setminus\{0,1\}$
and $p \in \Prime$ with $v_p(\alpha)=v_p(\alpha-1)=0$,
\begin{align*}
	F_p(\alpha)
	&\equiv 
	\sum_{k \in S_{p, 1}(\alpha)}
		\alpha^{\frac{p-1-2k\ord_{p}(\alpha)}{10}}
		\binom{I_{p}(\alpha)}{k}
	-
	\resi{\alpha}{p}\sum_{k \in S_{p, 2}(\alpha)}
		\binom{I_{p}(\alpha)}{k}
	\mod p
\end{align*}
holds, where
\begin{equation*}
	S_{p, i}(\alpha) 
	= 
	\{k \in \Z\mid 2k\ord_{p}(\alpha) \equiv p - i \mod 5\}
\end{equation*}
for $1\le i\le 4$.
\end{prop}
%\end{shadebox}
\begin{proof}
For $k\in \mathbb{Z}$, we have
\begin{align*}
	&\floor{(p - 1 - 5j)/2} = k\ord_{p}(\alpha)
	\textrm{\,\, for\,\,some\,\,} j 
	\\
	&\qquad\Leftrightarrow
	k\ord_{p}(\alpha) \le (p - 1 - 5j)/2 < k\ord_{p}(\alpha) + 1
	\textrm{\,\, for\,\,some\,\,} j 
	\\
	&\qquad\Leftrightarrow
	2k\ord_{p}(\alpha) = p - 1- 5j \textrm{\,\,or\,\,} p - 2 - 5j
	\textrm{\,\, for\,\,some\,\,} j 
	\\
	&\qquad\Leftrightarrow
	k \in S_{p, 1}(\alpha) \textrm{\,\,or\,\,}  k \in S_{p, 2}(\alpha).
\end{align*}
We note that
\begin{equation*}
	j 
	= 
	\frac{
		p - i - 2k\ord_{p}(\alpha)
	}{
		5
	}\in\mathbb{Z}
\end{equation*}
holds for $k\in S_{p,i}(\alpha)$.
Since we have

\begin{align*}
	\alpha^{\frac{j(5j + 1)}{2}} 
	= 
	\alpha^{
		\frac{
			(p - 1 - 2k\ord_{p}(\alpha))(p - 2k\ord_{p}(\alpha))
		}{
			10
		}
	}
	= 
	\left(
		\alpha^{p - 2k\ord_{p}(\alpha)}
	\right)^{
		\frac{
			p - 1 - 2k\ord_{p}(\alpha)
		}{
			10
		}
	}
	= 
	\alpha^{
		\frac{
			p - 1 - 2k\ord_{p}(\alpha)
		}{
			10
		}
	}
\end{align*}
for every $k \in S_{p, 1}(\alpha)$ and 
\begin{align*}
	\alpha^{\frac{j(5j + 1)}{2}} 
	&= 
	\alpha^{
		\frac{
			(p - 2 - 2k\ord_{p}(\alpha))(p - 1 - 2k\ord_{p}(\alpha))
		}{
			10
		}
	}
	= 
	\left(
		\alpha^{
			\frac{
				p - 1 - 2k\ord_{p}(\alpha)
			}{
				2
			}
		}
	\right)^{
		\frac{
			p - 2 - 2k\ord_{p}(\alpha)
		}{
			5
		}
	}
	\\
	&= 
	\resi{\alpha}{p}^{
		\frac{
			p - 2 - 2k\ord_{p}(\alpha)
		}{
			5
		}
	}
	= 
	\resi{\alpha}{p}
\end{align*}
for every $k \in S_{p, 2}(\alpha)$, we obtain
\begin{equation*}
	F_p(\alpha) 
	\equiv 
	\sum_{k \in S_{p, 1}(\alpha)}
		\alpha^{\frac{p-1-2k\ord_{p}(\alpha)}{10}}
		\binom{I_{p}(\alpha)}{k}
	-
	\resi{\alpha}{p}\sum_{k \in S_{p, 2}(\alpha)}
		\binom{I_{p}(\alpha)}{k} \mod p
\end{equation*}
by Theorem \ref{q-Fibonacci} and Lemma \ref{lem-qbinom}.
\end{proof}

% thm new fib p

%%%%%%%%%%%%%%%%%%%%%%%%%%%%%%%%Subsection2.3%%%%%%%%%%%%%%%%%%%%%%%%%%%%%%%%%%

\subsection{The certain congruence of the $q$-Fibonacci sequence}

In this subsection, we show our main theorem (Theorem \ref{thm-q-fib}).
\begin{lem}\label{lem-cong-q-fib}
Let $\alpha \in \Q\setminus\{0,1\}$
and $p \in \Prime$ with $v_p(\alpha)=v_p(\alpha-1)=0$.
If $\ord_p(\alpha)\not\equiv 0 \mod 5$, then
\[
F_p(\alpha)\equiv G_{I_p(\alpha),\ord_p(\alpha)} \mod p
\]
holds, where 
\[
G_{n,m}:=
(-1)^n\sum_{k\in 5\mathbb{Z}}\left\{\binom{n}{3n+k}-\binom{n}{3\left(n-\left(\frac{m}{5}\right)m\right)+k}\right\}
\]
for any $n$ and $m\in\mathbb{Z}$.
\end{lem}
\begin{rmk}
For $n$, $m\in\mathbb{Z}_{>0}$, $G_{n,m}$ is a finite sum. It is also noted $G_{n,0}=0$ in the case where $m\equiv 0 \mod 5$.  
\end{rmk}
\begin{proof}
Since $\ord_p(\alpha)\not\equiv 0 \mod 5$, we get
\begin{align*}
S_{p,1}=&\{k\mid 2k\ord_p(\alpha)\equiv p-1 \mod 5\}\\
=&\{k\mid 2k\ord_p(\alpha)\equiv I_p(\alpha)\ord_p(\alpha) \mod 5\}\\
=&\{k\mid k\equiv 3I_p(\alpha) \mod 5\}
\end{align*}
and
\begin{align*}
S_{p,2}=&\{k\mid 2k\ord_p(\alpha)\equiv p-2 \mod 5\}\\
=&\{k\mid 2k\ord_p(\alpha)\equiv I_p(\alpha)\ord_p(\alpha)-1\mod 5\}\\
=&\{k\mid k\equiv 3I_p(\alpha)-3\ord_p(\alpha)^{-1} \mod 5\}\\
=&\{k\mid k\equiv 3\left(I_p(\alpha)-\lsym{\ord_p(\alpha)}\ord_p(\alpha)\right) \mod 5\}.
\end{align*}
By using Proposition \ref{prop-old-qfib-p}, we have 
\begin{align*}
	F_p(\alpha)
	&\equiv 
	\sum_{k \in S_{p, 1}(\alpha)}
		\alpha^{\frac{p-1-2k\ord_{p}(\alpha)}{10}}
		\binom{I_{p}(\alpha)}{k}
	-
	\resi{\alpha}{p}\sum_{k \in S_{p, 2}(\alpha)}
		\binom{I_{p}(\alpha)}{k}
	\\
	&\equiv
	\sum_{k \in S_{p, 1}(\alpha)}
		\alpha^{\frac{\ord_{p}(\alpha)(I_p(\alpha)-2k)}{10}}
		\binom{I_{p}(\alpha)}{k}
	-
	\alpha^{\frac{\ord_{p}(\alpha)I_p(\alpha)}{2}}\sum_{k \in S_{p, 2}(\alpha)}
		\binom{I_{p}(\alpha)}{k} \mod p.
\end{align*}
\begin{enumerate}[i)]
\item\label{prop-q-fib-1}
In the case where $2 | I_{p}(\alpha)$, we have
\begin{align*}
	F_p(\alpha)
	&\equiv
	\sum_{k \in S_{p, 1}(\alpha)}
		\alpha^{\ord_{p}(\alpha)\frac{I_p(\alpha)-2k}{10}}
		\binom{I_{p}(\alpha)}{k}
	-
	\alpha^{\ord_{p}(\alpha)\frac{I_p(\alpha)}{2}}\sum_{k \in S_{p, 2}(\alpha)}
		\binom{I_{p}(\alpha)}{k}
	\\
	&\equiv
	\sum_{k \in S_{p, 1}(\alpha)}
		\binom{I_{p}(\alpha)}{k}
	-
	\sum_{k \in S_{p, 2}(\alpha)}
		\binom{I_{p}(\alpha)}{k}
	\\
	&\equiv 
	\sum_{k\in5\mathbb{Z}}
		\left\{
			\binom{I_p(\alpha)}{3I_p(\alpha)+k}-\binom{I_p(\alpha)}{3\left(I_p(\alpha)- \lsym{\ord_p(\alpha}\ord_p(\alpha)\right)+k}
		\right\} \mod p.
\end{align*}
\item
In the case where $2 \not| I_{p}(\alpha)$, we have
\begin{align*}
	F_p(\alpha)
	&\equiv
	\sum_{k \in S_{p, 1}(\alpha)}
		\alpha^{\frac{\ord_{p}(\alpha)}{2}\frac{I_p(\alpha)-2k}{5}}
		\binom{I_{p}(\alpha)}{k}
	-
	\alpha^{\frac{\ord_{p}(\alpha)}{2}I_p(\alpha)}\sum_{k \in S_{p, 2}(\alpha)}
		\binom{I_{p}(\alpha)}{k}
	\\
	&\equiv
	\sum_{k \in S_{p, 1}(\alpha)}
		(-1)^{\frac{I_p(\alpha)-2k}{5}}
		\binom{I_{p}(\alpha)}{k}
	-
	(-1)^{I_p(\alpha)}\sum_{k \in S_{p, 2}(\alpha)}
		\binom{I_{p}(\alpha)}{k}
	\\
	&\equiv
	-\sum_{k \in S_{p, 1}(\alpha)}
		\binom{I_{p}(\alpha)}{k}
	+
	\sum_{k \in S_{p, 2}(\alpha)}
		\binom{I_{p}(\alpha)}{k}\\
	&\equiv 
	-\sum_{k\in5\mathbb{Z}}
		\left\{
			\binom{I_p(\alpha)}{3I_p(\alpha)+k}-\binom{I_p(\alpha)}{3\left(I_p(\alpha)- \lsym{\ord_p(\alpha}\ord_p(\alpha)\right)+k}
		\right\} \mod p.
\end{align*}
\end{enumerate} 
\end{proof}

\begin{lem}\label{lem-G}
\begin{enumerate}[i)]
\item These
\[
G_{1,1}=1,\;G_{2,1}=1,\;G_{1,2}=0\text{ and }G_{2,2}=1
\]
hold.
\item Let $n\in\mathbb{Z}_{>0}$. For any $m$, $m'\in\mathbb{Z}$ satisfying $m\equiv \pm m' \mod 5$, 
\[
G_{n,m}=G_{n,m'}
\]
holds.
\item For any $n\in\mathbb{Z}_{>0}$ and $m\in\mathbb{Z}$,
\[
G_{n,m}+G_{n+1,m}=G_{n+2,m}
\]
holds. Therefore, if $m\not\equiv 0 \mod 5$, $G_{n,m}=F_{n+\lsym{m}}$ holds.
\end{enumerate}
\end{lem}
\begin{proof}
\begin{enumerate}[i)]
\item They are verified directly.
\item 
Since the following property of the finite summation
\begin{align}
\sum_{k\in 5\mathbb{Z}}\binom{m}{k-i}=\sum_{k\in 5\mathbb{Z}}\binom{m}{k+5-i}
\label{eq-finitesummation}
\end{align}
holds for $i=0$, $1$, $2$, $3$ and $4$, 
it is enough to show the case when $m'\equiv -m \mod 5$. Since 
\[
\lsym{1}=\lsym{-1},
\]
 we have
\begin{align*}
G_{n,m}
=&(-1)^n \sum_{k\in5\mathbb{Z}}\left\{\binom{n}{3n+k}-\binom{n}{3\left(n- \lsym{m}m\right)+k}\right\}\\
=&(-1)^n \sum_{k\in5\mathbb{Z}}\left\{\binom{n}{3n+k}-\binom{n}{-2n+ 3\lsym{m}m+k}\right\}\\
=&(-1)^n \sum_{k\in5\mathbb{Z}}\left\{\binom{n}{3n+k}-\binom{n}{3\left(n- \lsym{-m}(-m)\right)+k}\right\}\\
=&(-1)^n \sum_{k\in5\mathbb{Z}}\left\{\binom{n}{3n+k}-\binom{n}{3\left(n- \lsym{m'}m'\right)+k}\right\}\\
=&G_{n,m'}.
\end{align*}
\item By
Pascal's rule and (\ref{eq-finitesummation}), we have
\begin{align*}
&G_{n,m}+G_{n+1,m}-G_{n+2,m}\\
=&(-1)^n\sum_{k\in5\mathbb{Z}}\left\{
\binom{n}{3n+k}-\binom{n}{3n-3\lsym{m}m+k}
-\binom{n+1}{3n+3+k}\right. \\
&\left. +\binom{n+1}{(3n+3-3\lsym{m}m)+k}
-\binom{n+2}{3n+6+k}+\binom{n+2}{(3n+6-3\lsym{m}m)+k}
\right\}\\
=&(-1)^n\sum_{k\in5\mathbb{Z}}\left\{
\binom{n}{3n+k}-\binom{n}{3n-3\lsym{m}m+k}
-\binom{n}{3n+3+k}-\binom{n}{3n+2+k}\right. \\
& +\binom{n}{(3n-3\lsym{m}m+3+k}
+\binom{n}{(3n-3\lsym{m}m+2)+k}\\
& -\binom{n+1}{3n+6+k}
-\binom{n+1}{3n+5+k}\\
&\left. +\binom{n+1}{3n-3\lsym{m}m+1+k}
+\binom{n+1}{3n-3\lsym{m}m+k}\right\}\\
=&(-1)^n \sum_{k\in5\mathbb{Z}}\left(\sum_{0\le i\le 5} -\binom{n}{3n+i+k}+\binom{n}{3n-3\lsym{m}+i+k}\right)=0.
\end{align*}
\end{enumerate}
\end{proof}

By Lemma \ref{lem-cong-q-fib} and Lemma \ref{lem-G},
we conclude as follows:
%\begin{shadebox}
\begin{thm}[Theorem \ref{thm-q-fib}]
For $\alpha \in \Q^\times$
and $p \in \Prime$ satisfying $v_p(\alpha)=v_p(\alpha-1)=0$
and $\ord_{p}(\alpha) \not\equiv 0 \mod 5$,
\begin{equation*}
F_{p}(\alpha)\equiv F_{I_p(\alpha)+\lsym{\ord_p(\alpha)}} \mod p
\end{equation*}
holds.
\end{thm}

%%%%%%%%%%%%%%%%%%%%%%%%%%%%%%%%%%%%Section3%%%%%%%%%%%%%%%%%%%%%%%%%%%%%%%%%%%

\section{Finite algebraic numbers and the $\A$-transcendence of $q$-Fibonacci sequence}
In this section, we show that, under GRH, the element of $\A$ associated with the $q$-Fibonacci sequence is an $\A$-transcendental number (in Theorem \ref{thm-trans2}).

%%%%%%%%%%%%%%%%%%%%%%%%%%%%%%%%%%%%Subsection3.1%%%%%%%%%%%%%%%%%%%%%%%%%%%%%%

\subsection{Review on finite algebraic numbers}
\verb| |

This subsection discusses finite algebraic numbers introduced in \cite{Rosen1} which are defined by recurrent sequences. 

For $(\alpha_p)_p\in \prod_{p} \mathbb{Z}/p\mathbb{Z}$ and $n\in\mathbb{Z}_{>0}$, we define 
\[
P_n((\alpha_p)):=\{p\mid \alpha_p\equiv n \mod p\}\subset P.
\]
By abuse of notation, 
the image of $\displaystyle \alpha=(\alpha_p)_p\in \prod_p \mathbb{Z}/p\mathbb{Z}$
under the natural projection to $\A$ is denoted to be $\alpha$.
We say $n$ \textit{occurs infinitely} often in $\alpha$ if $|P_n((\alpha_p)_p)|$ is infinite (we note that it is independent of any choice of representatives in $\displaystyle \prod_p \mathbb{Z}/p\mathbb{Z}$). 

A sequence $(a_n)_n\subset \mathbb{Q}$ is called \textit{recurrent} if there exists a monic polynomial $f(x):=x^d+c_1x^{d-1}+\cdots+c_d\in\mathbb{Q}[x]$ such that
\[
a_{n+d}+c_1a_{n+d-1}+\cdots+c_da_n=0\;\;(\text{on $\mathbb{Q}$})
\]
for every $n\in\mathbb{N}$. Such $f$ and $(a_0,\ldots,a_{d-1})$ is called the characteristic polynomial and the initial value of $(a_n)_n$ respectively.   
\begin{df}[\cite{Rosen1}]
An element $(\alpha_p)_p\in \A$ is called a \textit{finite algebraic number} if there exists a recurrent sequence $(a_n)_n$ such that 
\[
\alpha_p\equiv a_p \mod p
\]
for sufficiently every large $p$. Let $\mathcal{P}_{\A}^0\subset \A$ be the set of finite algebraic numbers.
\end{df}
% denotes an integral closure of $\A$ over $\mathbb{Q}$.
%We note that whole finite algebraic numbers $\mathcal{P}_0^{\A}$ are contained in $\Qi$.

The finite algebraic numbers are characterized in \cite{Rosen1}.
\begin{thm}[{\cite[Theorem 1.1]{Rosen1}}]\label{thm-finalg}
Let $\alpha:=(\alpha_p)_p\in \A $. Then the following conditions are equivalent:
\begin{enumerate}
\item  The element $\alpha\in \A$ is a finite algebraic number.
\item There exists a Galois extension $L/\mathbb{Q}$ and a map $\phi:\Gal(L/\mathbb{Q})\rightarrow L$ satisfying $\phi(\sigma\tau\sigma^{-1})=\sigma \phi(\tau)$ for $\sigma$, $\tau\in \Gal(L/\mathbb{Q})$ such that
\[
(\alpha_p)_p=(\phi(\mathfrak{F}_p) \mod p)_p,
\]
where $\mathfrak{F}_p$ is a Frobenius map of $L$ at prime $p$.
\end{enumerate}
\end{thm}

\begin{rmk}
\begin{enumerate}
\item The set $\mathcal{P}_{\A}^0$ is a $\mathbb{Q}$-subalgebra of $\A$ in \cite{Rosen1}.
\item 
\cite[$\S4$]{Rosen1} explains that the $\mathbb{Q}$-algebra $\mathcal{P}_{\A}^0$ is regarded as a finite analogue of an algebraic closure $\overline{\mathbb{Q}}$ of $\mathbb{Q}$.

\end{enumerate}
\end{rmk}

Let $\Qi:=\{\alpha\in \A\mid {}^{\exists} f(X)\in\mathbb{Q}[X]\setminus\{0\}\text{ such that }f(\alpha)=0 \}$.
\begin{thm}[{\cite[Theorem 1.4]{Rosen1}}]\label{integral-closed}
\begin{enumerate}
\item\label{integral-closed1} For each $\alpha\in \mathcal{P}_{\A}^0$, there exists a polynomial $f(x)\in\mathbb{Q}[x]^\times$ such that $f(\alpha)=0$ in $\A$.
\item There is a sequence of inclusions of subsets:
\[
\mathbb{Q}\subsetneq \mathcal{P}_{\A}^0\subsetneq \Qi\subset \A.
\]
\end{enumerate}
\end{thm}
\begin{rmk}
Note that $\mathcal{P}_\A^0$ is countable and $\Qi$ is uncountable (\cite{Rosen1}).
\end{rmk}
\begin{df}\label{df-ftn}
%\begin{enumerate}
%\item
%An element of $\Qi\setminus \mathcal{P}_0^{\A}$ is called a finite quasi-transcendental number.
%\item
An element of $\A\setminus \Qi$ is called an \textit{$\A$-transcendental number}.
%\end{enumerate}
\end{df}
FMZV is expected to be an $\A$-transcendental number when it is non-zero, but a single example has not been found. According to the $\A$-analogue of the period conjecture proposed by Rosen (\cite{Rosen2018}), at least FMZV of depth $2$ is expected to be non-zero.
\begin{prop}\label{dm}
Let $\alpha \in \ca{A}$. If there exists a sequence of distinct integers $(a_n)_n$ such that $a_n$ occurs infinitely often in $\alpha$ for every $n\in\mathbb{N}$, then $\alpha$ is an $\A$-transcendental number.
\end{prop}

\begin{proof}

Suppose that $\alpha\in \mathcal{C}_{\A}^0$.
By the definition of $\Qi$, we can take $f(x)\in \mathbb{Q}[x]\setminus\{0\}$ with $f(\alpha)=0$ as an element in $\A$.
Since the components of $(a_n)_n$ are distinct, 
we can take $n\in\mathbb{N}$ such that
\[
f(a_n)\neq 0
\]
in $\mathbb{Q}$. This implies that 
\[
f(a_n)\not\equiv 0 \mod p
\]
for sufficiently every large $p$. Since $a_n$ occurs infinitely often in $\alpha$ for every $n\in\mathbb{N}$,
\[
f(\alpha)\neq 0
\]
holds on $\A$. It contradicts that $f(\alpha)=0$ in $\A$.

\end{proof}

\begin{ex}
Let $\alpha\in \A$ be represented by
$(\overbrace{1}^1,\overbrace{1,2}^2,\overbrace{1,2,3}^3,\ldots )$.
Then every $m\in\mathbb{Z}_{>0}$ occurs infinitely often in $\alpha$. By Proposition \ref{dm}, this $\alpha$ is an $\A$-transcendental number.
Hence $\mathcal{C}_{\A}^0\subsetneq \A$ holds.
\end{ex}

%%%%%%%%%%%%%%%%%%%%%%%%%%%%%%%%%%%%Subsection3.2%%%%%%%%%%%%%%%%%%%%%%%%%%%%%%

\subsection{The $\A$-transcendence of the $q$-Fibonacci sequence}\verb| |

This subsection gives the proof of Theorem \ref{thm-trans} (Theorem \ref{thm-trans2}).

By Theorem \ref{thm-q-fib}, the $q$-Fibonacci sequence is related with the residual indices.
Several results on the residual indices are obtained under the following conjecture called the generalized Riemann hypothesis (GRH, in short).

\begin{conj}\label{Riemann}
The real part of every non-trivial zero of the Dedekind zeta function of an algebraic field $K$ is equal to $\displaystyle \frac{1}{2}$.
\end{conj}
\begin{rmk}
In this paper, we assume the GRH for all fields $K_{s,r}^g=\mathbb{Q}(\zeta_s, g^{\frac{1}{r}})$, where $g$ is a positive square-free integer,  $s$ and $r$ are integers satisfying $r|s$ and $\zeta_u$ is a $u$-th root of the unity.
%The studies of the residual index presented below are under the family of GRH for the above form of fields. 
\end{rmk}

Under GRH, Hooley (\cite{Hooley}) calculated the density of primes $p$ satisfying $I_p(\alpha)=1$ for a square-free positive integer $\alpha$.
Under GRH, Lenstra (\cite{Lenstra}) calculated the density of primes satisfying $I_p(\alpha)=k$ for a square-free integer $\alpha$ and a positive integer $k$. 
He also calculated the condition that the density of primes satisfying $I_p(\alpha)=k$ is equal to $0$.
Murata (\cite{Murata}) showed the asymptotic formulae on such primes less than a given positive real number $x$ for a given square-free integer $\alpha$.

%To prove Corollary \ref{cor-trans}, we prepare the following lemma.

We prepare some notations to prove Theorem \ref{thm-trans2} (Theorem \ref{thm-trans}).
Let $v$ and $s\in\mathbb{Z}_{>0}$.
% let $s$, $r\in\mathbb{Z}_{>0}$ satisfying $r|s$, 
Let $\zeta_s$ be a root of unity and $\sigma_b:\mathbb{Q}(\zeta_s)\rightarrow \mathbb{Q}(\zeta_s),\;\zeta_s\mapsto \zeta_s^b$ for $b\in\mathbb{Z}$ satisfying $(s,b)=1$.
\begin{itemize}
%\item $K_{s,r}^g:=\mathbb{Q}(\zeta_s,g^{\frac{1}{r}})$
\item Let $[a,b]$ be the least common multiple of $a$ and $b$ and $(a,b)$ be the greatest common divisor of $a$ and $b$.
\item $C_g(b,f,v)=
\begin{cases}
1&\text{if }\sigma_b|_{\mathbb{Q}(\zeta_f)\cap K_{v,v}^g}=\id\\
0&\text{otherwise}.
\end{cases}$
\item Let $\mu:\mathbb{Z}_{>0}\rightarrow\{0,\pm1\}$ be the M\"{o}bius function, i.e.
\[
\mu(n):=
\begin{cases}
0& \text{if $n$ has a squared prime factor}\\
(-1)^{k}& \text{if $n$ is the product of $k$ distinct primes}.
\end{cases}
\]
\end{itemize}
Moree showed the following lemma under GRH.

\begin{lem}[{\cite[Lemma 11]{MoreeII}}]\label{Lem-Moree-detail}

Let $a$, $d$, $t\in\mathbb{Z}_{>0}$, let $g\in\mathbb{Z}\setminus\{0,\pm1\}$ be square-free and $x\in\mathbb{R}_{>0}$.
Put
\[
V_g(a,d;t)(x):=|\{p\le x\mid p\in P,\;I_p(g)=t, p\equiv 1+ta \mod dt\}|.
\]
%For $t<x^{\frac{1}{3}}$,
For sufficiently large $x$,
\[
V_g(a,d;t)(x)=\frac{x}{\log x}\delta(a,d;t)
+O_{g,d}\left(\frac{x\log \log x}{\varphi(t)\log^2 x}+\frac{x}{\log ^2x}\right)
\]
holds under GRH, where 
\[ \delta(a,d;t):=\sum_{\substack{n=1\\ (n,d)|a}}^\infty \frac{\mu(n)C_g(1+ta;dt;nt)}{[K_{[d,n]t,nt}^g:\mathbb{Q}]},
\]
 $\varphi$ is Euler's totient function and $O_{g,d}$ is the Landau notation with respect to $g$ and $d$.
\end{lem}

\begin{prop}\label{prop-Moree}
Under GRH for all such fields $K_{s,r}^g$, if $t\in (5\mathbb{Z}_{>0}+1)\cap P$, $a\in\mathbb{Z}_{>0}$ and $g$ is a square-free integer, then there are infinitely many primes $p$ satisfying $I_p(g)=t$ and  $p\equiv 1+t \mod 5t$.
\end{prop}
We set up the following auxiliary lemmas to prove the above proposition.
\begin{lem}[{\cite[Equation (12)]{Hooley}}]
Let $g$ be a square-free integer. Let $s$ and $r$ be positive integers satisfying $r|s$.
Then we have
\begin{align}
[K_{s,r}^g:\mathbb{Q}]=\frac{r\varphi(s)}{\varepsilon_g(s)},
\label{eq-deg}
\end{align}
where
\[
\varepsilon_g(s):=
\begin{cases}
2&\text{if $2g|s$ and $g\equiv 1 \mod 4$}\\
1&\text{otherwise}.
\end{cases}
\] 
\end{lem}
\begin{lem}\label{lem-dim}
Let $a\ge 2$, $b\ge 2$ and $p$ be an odd prime. Let $g\in\mathbb{Z}\setminus\{0,\pm1\}$ be a square-free  integer.
\begin{enumerate}
\item\label{lem-dim1} 
If $b|a$,  then we have $\displaystyle [\mathbb{Q}(\zeta_{ap}, g^{\frac{1}{b}}):\mathbb{Q}(\zeta_a,g^{\frac{1}{b}})]\ge \frac{p-1}{2}$.
\item\label{lem-dim2} 
If $bp|a$, then $[\mathbb{Q}(\zeta_{a},g^{\frac{1}{bp}}):\mathbb{Q}(\zeta_{a},g^{\frac{1}{b}})]=p$ holds.
\end{enumerate}
\end{lem}
\begin{proof}

(\ref{lem-dim1})
By (\ref{eq-deg}),  we have
\[
[\mathbb{Q}(\zeta_{ap}, g^{\frac{1}{b}}):\mathbb{Q}(\zeta_a,g^{\frac{1}{b}})]
=[\mathbb{Q}(\zeta_{ap}, g^{\frac{1}{b}}):\mathbb{Q}]/[\mathbb{Q}(\zeta_a,g^{\frac{1}{b}}):\mathbb{Q}]
=\frac{\frac{b\varphi(ap)}{\varepsilon_g(ap)}}{\frac{b\varphi(a)}{\varepsilon_g(a)}}
=\frac{\varphi(ap)}{\varphi(a)}\cdot\frac{\varepsilon_g(a)}{\varepsilon_g(ap)}
\ge\frac{\varphi(ap)}{2\varphi(a)}.
\]
\begin{itemize}
\item
Suppose $(a,p)=1$. Since $\displaystyle \frac{\varphi(ap)}{\varphi(a)}=\varphi(p)=p-1$, we have $\displaystyle [\mathbb{Q}(\zeta_{ap}, g^{\frac{1}{b}}):\mathbb{Q}(\zeta_a,g^{\frac{1}{b}})]\ge \frac{p-1}{2}$.

\item
Suppose $(a,p)=p$ and take  $a=cp^k$ $(c\in\mathbb{Z}_{>0})$. Since $\displaystyle \frac{\varphi(c)\varphi(p^{k+1})}{\varphi(c)\varphi(p^k)}=p$, we have $\displaystyle [\mathbb{Q}(\zeta_{ap}, g^{\frac{1}{b}}):\mathbb{Q}(\zeta_a,g^{\frac{1}{b}})]=\frac{p}{2}\ge \frac{p-1}{2}$.
\end{itemize}

(\ref{lem-dim2})
The equation (\ref{eq-deg}) shows
\[
[\mathbb{Q}(\zeta_a,g^{\frac{1}{bp}}):\mathbb{Q}(\zeta_a,g^{\frac{1}{b}})]
=[\mathbb{Q}(\zeta_a,g^{\frac{1}{bp}}):\mathbb{Q}(\zeta_a)]/[\mathbb{Q}(\zeta_{a},g^{\frac{1}{b}}):\mathbb{Q}(\zeta_a)]
=\frac{bp}{\varepsilon_g(a)}\cdot \frac{\varepsilon_g(a)}{b}
=p.
\]

\end{proof}
\begin{proof}[Proof of Proposition \ref{prop-Moree}]
By Lemma \ref{Lem-Moree-detail}, it is sufficient to show $\delta(1,5;t)>0$.

If we consider the prime factorization of $n$ and that $(n, 5) | 1$ and $(n, 5) = 1$ are equivalent, then $\delta(1, 5; t)$ can be written as follows:
\begin{align*}
\delta(1,5;t)&=\sum_{k \ge 0}(-1)^k \sum_{\substack{p_1<\cdots<p_t\\ p_i\neq 5}}\frac{C_g(1+t,5t,p_1\cdots p_k t)}{[\mathbb{Q}(\zeta_{5p_1\cdots p_k t},g^{\frac{1}{p_1\cdots p_{k}t}}):\mathbb{Q}]}\\
&=\sum_{k\in 2\mathbb{Z}_{\ge 0}}\left(\sum_{\substack{p_1<\cdots<p_k\\ p_i\neq 5}}\frac{C_g(1+t,5t,p_1\cdots p_k t)}{[\mathbb{Q}(\zeta_{5p_1\cdots p_k t},g^{\frac{1}{p_1\cdots p_{k}t}}):\mathbb{Q}]}
-\sum_{\substack{p_1<\cdots<p_{k+1}\\ p_i\neq 5}}\frac{C_g(1+t,5t,p_1\cdots p_{k+1} t)}{[\mathbb{Q}(\zeta_{5p_1\cdots p_{k+1} t},g^{\frac{1}{p_1\cdots p_{k+1}t}}):\mathbb{Q}]}\right).
\end{align*}
Note that if $\sigma_b|_{\mathbb{Q}(\zeta_f)\cap K_{v,v}}\neq \id$, then $\sigma_b|_{\mathbb{Q}(\zeta_f)\cap K_{pv,pv}}\neq \id$ for $b$, $f$ and $v\in\mathbb{Z}_{>0}$ and a prime number $p$. 
This implies if $C_g(1+t,5t,p_1\cdots p_k t)=0$, then $C_g(1+t,5t,p_1\cdots p_{k+1} t)=0$ holds.
Moreover, $C_g(1+t,5t,p_1\cdots p_k t)C_g(1+t,5t,p_1\cdots p_{k+1} t)=C_g(1+t,5t,p_1\cdots p_{k+1} t)$ holds.
Thus we have
\begin{align*}
\delta(1,5;t)&\\
=&\sum_{k\in 2\mathbb{Z}_{\ge 0}}\left(\sum_{\substack{p_1<\cdots<p_k\\ p_i\neq 5}}\frac{C_g(1+t,5t,p_1\cdots p_k t)}{[\mathbb{Q}(\zeta_{5p_1\cdots p_k t},g^{\frac{1}{p_1\cdots p_{k}t}}):\mathbb{Q}]}\right. \\%3-1
&\left.
-\sum_{\substack{p_1<\cdots<p_{k+1}\\ p_i\neq 5}}\frac{C_g(1+t,5t,p_1\cdots p_{k+1} t)C_g(1+t,5t,p_1\cdots p_{k} t)}{[\mathbb{Q}(\zeta_{5p_1\cdots p_{k+1} t},g^{\frac{1}{p_1\cdots p_{k+1}t}}):\mathbb{Q}(\zeta_{5p_1\cdots p_{k}t},g^{\frac{1}{p_1\cdots p_k t}})][\mathbb{Q}(\zeta_{5p_1\cdots p_k t},g^{\frac{1}{p_1\cdots p_{k}t}}):\mathbb{Q}]}\right)
\\%3-2
=&\sum_{k\in 2\mathbb{Z}_{\ge 0}}\sum_{\substack{p_1<\cdots<p_k\\ p_i\neq 5}}\frac{C_g(1+t,5t,p_1\cdots p_k t)}{[\mathbb{Q}(\zeta_{5p_1\cdots p_k t},g^{\frac{1}{p_1\cdots p_{k}t}}):\mathbb{Q}]}\times \\%4-1
&\left(1
-\sum_{\substack{p_k<p_{k+1}\\ p_{k+1}\neq 5}}\frac{C_g(1+t,5t,p_1\cdots p_{k+1} t)}{[\mathbb{Q}(\zeta_{5p_1\cdots p_{k+1} t},g^{\frac{1}{p_1\cdots p_{k+1}t}}):\mathbb{Q}(\zeta_{5p_1\cdots p_{k}t},g^{\frac{1}{p_1\cdots p_k t}})]}
\right).%4-2
\end{align*}
Hence it is sufficient to show
\begin{align}
\sum_{\substack{p_k<p_{k+1}\\ p_i\neq 5}}\frac{C_g(1+t,5t,p_1\cdots p_{k+1} t)}{[\mathbb{Q}(\zeta_{5p_1\cdots p_{k+1} t},g^{\frac{1}{p_1\cdots p_{k+1}t}}):\mathbb{Q}(\zeta_{5p_1\cdots p_{k}t},g^{\frac{1}{p_1\cdots p_k t}})]}
<1
\label{eq-den}
\end{align}
for every $k\in2\mathbb{Z}_{\ge 0}$.
By Lemma \ref{lem-dim}, 
\begin{align*}
&[\mathbb{Q}(\zeta_{5p_1\cdots p_{k+1}t},g^{\frac{1}{p_1\cdots p_{k+1}t}})
:\mathbb{Q}(\zeta_{5p_1\cdots p_{k}t},g^{\frac{1}{p_1\cdots p_k t}})]\\
=&
[\mathbb{Q}(\zeta_{5p_1\cdots p_{k+1}t},g^{\frac{1}{p_1\cdots p_{k+1}t}})
:\mathbb{Q}(\zeta_{5p_1\cdots p_{k+1}t},g^{\frac{1}{p_1\cdots p_{k}t}})]
\cdot
[\mathbb{Q}(\zeta_{5p_1\cdots p_{k+1}t},g^{\frac{1}{p_1\cdots p_{k}t}})
:\mathbb{Q}(\zeta_{5p_1\cdots p_{k}t},g^{\frac{1}{p_1\cdots p_{k}t}})]\\
\ge& p_{k+1}\frac{(p_{k+1}-1)}{2}\ge \frac{(p_{k+1}-1)^2}{2}
\end{align*}
holds. Therefore,
\begin{enumerate}[(i)]
\item In the case where $k>0$,
we have
\begin{align*}
\sum_{\substack{p_k<p_{k+1}\\ p_i\neq 5}}\frac{C_g(1+t,5t,p_1\cdots p_{k+1} t)}{[\mathbb{Q}(\zeta_{5p_1\cdots p_{k+1} t},g^{\frac{1}{p_1\cdots p_{k+1}t}}):\mathbb{Q}(\zeta_{5p_1\cdots p_{k}t},g^{\frac{1}{p_1\cdots p_k t}})]}
&<2\sum_{\substack{p_k<p_{k+1} \\ p_k\neq 5}}\frac{1}{(p_{k+1}-1)^2}\\
&\le 2\sum_{l\in\mathbb{Z}_{>0}}\frac{1}{4l^2}\le \frac{\pi^2}{12}<\frac{5}{6}<1.
\end{align*}

\item In the case where $k=0$, since we have
\[
[\mathbb{Q}(\zeta_{5pt},g^{\frac{1}{pt}}):\mathbb{Q}(\zeta_{5t},g^{\frac{1}{t}})]
= \frac{\frac{pt\varphi(5pt)}{\varepsilon_g(5pt)}}{\frac{t\varphi(5t)}{\varepsilon_g(5t)}}
\ge \frac{5p\varphi(pt)}{2}
\ge \frac{5p(p-1)}{2}
\ge \frac{5(p-1)^2}{2}
\]
by (\ref{eq-deg}), we have
\begin{align*}
\sum_{\substack{p\in P\\ p\neq 5}} \frac{C_g(1+t,5t,pt)}{[\mathbb{Q}(\zeta_{5pt},g^{\frac{1}{pt}}):\mathbb{Q}(\zeta_{5t},g^{\frac{1}{t}})]}
\le&\;\frac{2}{5}\sum_{p\in P} \frac{1}{(p-1)^2}\le \frac{\pi^2}{15}<\frac{2}{3}.
\end{align*}
\end{enumerate}
So (\ref{eq-den}) is obtained. Hence we show the claim.
\end{proof}
Using the above lemma, we obtain the result on $\A$-transcendental numbers.
\begin{thm}[Theorem \ref{thm-trans}]\label{thm-trans2}
Let $g\in\mathbb{Z}_{>1}$ be a square-free integer.  
Then $F=(F_p(g))_p$ is an $\A$-transcendental number under GRH for all such fields $K_{s,r}^g$.
\end{thm}
\begin{proof}
Let $t\in (5\mathbb{Z}_{>0}+1)\cap P$. By Proposition \ref{prop-Moree}, there are infinitely many primes $p$ which satisfy $I_p(g)=t$ and  $p\equiv 1+t \mod 5t$.
Therefore there are infinitely many primes $p$ which satisfy $I_p(g)=t$ and  $p\equiv 2 \mod 5$ by $t\equiv 1 \mod 5$.
By Theorem \ref{thm-q-fib}, at least
one of $F_{t-1}$ or $F_{t+1}$ occurs infinitely often in $F$.
That means there exists a subsequence $(a_n)$ of the Fibonacci sequence such that 
$a_n$ occurs infinitely on $F$
for every $n\in\mathbb{N}$. Hence $F\notin \Qi$ by Proposition \ref{dm}.
\end{proof}

%%%%%%%%%%%%%%%%%%%%%%%%%%%%%%%%%%%%%%

%%%%%%%%%%%%%%%%%%%%%%%%%%%%%%%%%%%%Appendix%%%%%%%%%%%%%%%%%%%%%%%%%%%%%%%%%%
\begin{comment}
\appendix

\section{Finite multiple zeta values and its finite transcendence}
In this section, we describe the finite multiple zeta values and their finite transcendence under the $\A$-analogue of the Grothendieck's period conjecture.
\subsection{Finite multiple zeta values}
We define the finite multiple zeta values.
\begin{df}
Let $r\in\mathbb{Z}_{>0}$. For $\mathbf{k}:=(k_1,\ldots,k_r)\in\mathbb{Z}_{>0}^r$, the finite multiple zeta values are the truncated summation of the original multiple zeta values in $\A$ defined by
\[
\fz(\mathbf{k}):=\left(\sum_{p>m_1>\cdots>m_r>0} \frac{1}{m_1^{k_1}\cdots m_r^{k_r}} \mod p \right)\in \A.
\]
\end{df}
FMZV has the expression of the multi-poly-Bernoulli numbers.
\begin{thm}[\textcolor{red}{write reference}]
For every $\mathbf{k}:=(k_1,\ldots,k_r)\in\mathbb{Z}^r$,
\[
\fz(\mathbf{k})=(-C_{p-2}^{k_1-1,k_2,\ldots,k_r})_p \in \A
\] 
holds where $C_{n}^{k_1,\ldots,k_r}$ is the multi-poly-Bernoulli numbers defined by
\[
\frac{\Li_{k_1,\ldots,k_r}(1-e^{-t})}{1-e^{-t}}=\sum_{n=0}^\infty \frac{C_n^{(k_1,\ldots,k_r)}}{n!}t^n
\]
for each $n\in\mathbb{Z}_{\ge 0}$ and $(k_1,\ldots,k_r)\in\mathbb{Z}^r$, and $\Li_{k_1,\ldots,k_r}(z)$ is the multiple polylogarithm series given by
\[
\Li_{k_1,\ldots,k_r}(z):=\sum_{m_1>\ldots>m_r>0} \frac{z^{m_1}}{m_1^{k_1}\cdots m_r^{k_r}}.
\]
\end{thm} 
\section{The $\A$-analog of the Grothendieck's period conjecture}
\end{comment}

%%%%%%%%%%%%%%%%%%%
%%%%  ref   %%%%
%%%%%%%%%%%%%%%%%%%
\begin{center}
Acknowledgments
\end{center} 

This work was financially supported by JST SPRING, Grant Number JPMJSP2125. T.A. would like to take this opportunity to thank the “Interdisciplinary Frontier Next-Generation Researcher Program of the Tokai Higher Education and Research System".
We are deeply grateful to  Hidekazu Furusho; without his profound instruction and continuous encouragement, this paper would never be accomplished.
We are grateful to Henrik Bachmann, Minoru Hirose, Toshiki Matsusaka, Leo Murata, Shin-ichiro Seki, Koji Tasaka and Shuji Yamamoto for answering our questions. 
We would like to thank  Jun Ueki for giving us much advice on the structure of this paper.

\bibliographystyle{amsplain}
\bibliography{C:/Users/User/Documents/Tex/reference/transcendental}
\end{document}